\documentclass{amsproc}

\usepackage{amssymb}
\usepackage{esint}

\newcommand{\R}{\mathbb{R}}

\newcommand{\Z}{\mathbb{Z}}

\newcommand{\bes}{\dot B^{-N/p^*,\infty,\infty}(\R^N)}

\newcommand{\sob}{\dot H^{1,p}(\R^N)}
\newcommand{\sobemb}{\dot H^{1,p}(\R^N)\hookrightarrow L^{p^*}(\R^N)}
\newcommand{\leb}{L^{p^*}(\R^N)}

\newtheorem{theorem}{Theorem}[section]
\newtheorem{lemma}[theorem]{Lemma}

\theoremstyle{definition}
\newtheorem{definition}[theorem]{Definition}

\theoremstyle{remark}

\numberwithin{equation}{section}

\begin{document}

\title{Four proofs of cocompacness for Sobolev embeddings}


\author{Cyril Tintarev}
\address{Uppsala University}
\curraddr{}
\email{tintarev@math.uu.se}
\thanks{}


\date{}

\begin{abstract}
Cocompactness is a property of embeddings between two Banach spaces, similar to but weaker than compactness, defined relative to some non-compact group of bijective isometries. In presence of a cocompact embedding, bounded sequences (in the domain space) have subsequences that can be represented as a sum of a well-structured``bubble decomposition'' (or {\em defect of compactness}) plus a remainder vanishing in the target space.
This note is an exposition of different proofs of cocompactness for Sobolev-type embeddings, which employ methods of classical PDE, potential theory, and harmonic analysis.
\end{abstract}

\maketitle
\section{Introduction}
Cocompactness is a property of continuous embeddings between two Banach spaces, defined as follows.
\begin{definition}\label{def:coco} 
Let $X$ be a reflexive Banach space continuously embedded into a Banach space $Y$. The embedding $X\hookrightarrow Y$ is called cocompact relative to a  group $G$ of bijective isometries on $X$ if every sequence $(x_k)$, such that $g_kx_k\rightharpoonup 0$ for any choice of $(g_k)\subset G$, converges to zero in $Y$.  
\end{definition}
In particular, a compact embedding is cocompact relative to the trivial group $\{I\}$ (and therefore to any other group of isometries), but the converse is generally false. In this note we study embedding of the Sobolev space $\sob$, defined as the completion of $C_0^\infty(\R^N)$ with respect to the gradient norm $(\int|\nabla u|^pdx)^\frac{1}{p}$, $1<p<N$, into the Lebesgue space $L^\frac{Np}{N-p}(\R^N)$. This embedding is not compact, but it is cocompact relative to the rescaling group (see \eqref{cco} below). 

An elementary example of a cocompact embedding that is not compact, due to Jaffard \cite{Jaffard}, is the embedding of $\ell^\infty(\mathbb Z)$ into itself, which is cocompact relative to the group of shifts $\{g_n:\;\{x_j\}\mapsto \{x_{j-n}\}, n\in\Z\}$. Indeed, assume that $g_{n_k}x^{(k)}\rightharpoonup 0$ for any choice of $(n_k)$. Choose a component $x^{k}_{j_k}$ of $x^{(k)}$ so that $\|x^{(k)}\|_\infty\le 2|x^{k}_{j_k}|$. Then, applying the sequence $g_{j_k}$ we have $\|x^{(k)}\|_\infty\le 2|\langle g_{j_k}x^{(k)},e_0\rangle|\to 0$. In fact, this type of argument can be used to prove cocompactness of any embedding, where the spaces involved have an unconditional wavelet basis associated with the group, as it is indeed done in \cite{Jaffard} for fractional Sobolev spaces, and in \cite{BCK} for Besov and Triebel-Lizorkin spaces.

 Cocompactness is a tool widely used (often without being referred to by any name, or else as ``vanishing lemma`` or ``inverse embedding``) for proving convergence of functional sequences, in particular in calculus of variations. Cocompactness property implicitly underlies the concentration compactness method of P.-L. Lions, but it is important to stress that the latter was developed for particular function spaces, while cocompactness argument is functional-analytic, and its applications extend beyond the concentration phenomena (see \cite{survey}). 
 The earliest cocompactness result known to the author is the proof by E.~Lieb \cite{Lieb} of cocompactness of embedding of the inhomogenous Sobolev space $H^{1,p}$ into $L^{q}$, $q\in(p,p^*)$, relative to the group of shifts $u\mapsto u(\cdot-y)\;y\in\R^N$ (which is an easy consequence of the cocompactness result discussed here). We refer the reader to a survey  \cite{survey} of known cocompact embeddings for function spaces. 
 The purpose of this note is to complement that survey of results with four different proofs of cocompactness for Sobolev embeddings. We generalize the second and the fourth proofs found in literature for the  case $p=2$ to general $p$, and  the fourth proof is shortened in comparison to the original version by means of referring to a known inequality. The third and the fourth proofs can be easily extended ti embeddings of Besov and Tribel-Lizorkin spaces, as it is done in, respectively \cite{BCK} and (to our best knowledge) \cite{Cwikel}. 

Cocompactness property plays crucial role in describing the ''blow-up`` structure of sequences. In presence of a $G$-cocompact embedding $X\hookrightarrow Y$, any bounded sequence in $X$, under suitable additional conditions, has a subsequence consisting of asymptotically decoupled ''bubbles`` of the form $g_kw$ with $g_k\in G$ and $w\in X$, plus a remainder that vanishes in the norm of $Y$. Existence of such decomposition was first proved by Solimini \cite{Solimini} for the Sobolev embedding $\sobemb$, generalized by several subsequent papers to other spaces of Sobolev type, to more general groups, as well as to Strichartz spaces (see survey \cite{survey}), extended to general Hilbert spaces in \cite{SchinTin}, and to general Banach spaces in \cite{SoliTi}. The latter work, however, uses a related, but different, property of Delta-cocompactness, which can be formulated by replacing the weak convergence in the statement of Definition~1.1 with Delta-convergence of  T.~C.~Lim \cite{Lim}.   
\begin{definition}
 Let $X$ be a metric space. One says  a sequence $(x_n)\subset X$ has  a Delta-limit $x$ if for any $y\in X$ there is a sequence of real numbers $\alpha_n\to 0$, such that 
 \[
 d(x_n,y) \ge d(x_n,x)-\alpha_n.
 \]
\end{definition}
Delta-convergence and weak convergence coincide in Hilbert spaces, but not in general Banach spaces. Delta-convergence is, however, dependent on a norm, and in separable spaces one can always find an equivalent norm, relative to which Delta-convergence coincides with weak convergence (van Dulst, \cite{vanDulst}). Furthermore, for Besov and Triebel-Lizorkin (including Sobolev) spaces, such equivalent norm exists while it also remains invariant with respect to the rescaling group, defined via the Littlewood-Paley decomposition (\cite{Triebel}, Definition 2 in Chapter 5), see Cwikel, \cite{Cwikel}. Thus, in the context of spaces of Sobolev type, there is no need to consider Delta-cocompactness as a property distinct from cocompactness. 

We do not include here the case $p=1$, because the natural extension of the notion of cocompactness to a non-reflexive space would invoke the weak-star convergence, which is not defined on $\dot H^{1,1}$. Instead one may regard $\dot H^{1,1}(\R^N)$ as a closed isometric subspace of the space of functions of bounded variation, and prove cocompactness of the embedding $BV(\R^N)\hookrightarrow L^{1^*}(\R^N)$ for $N\ge 2$, as it is done in \cite{AT-BV}. The proof here is an adaptation of the second proof in this note that takes into account a different chain rule for functions of bounded variation.  

A larger collection of cocompact embeddings can be produced by restricting homogeneous functional spaces to their subspaces, in particular to inhomogenous Sobolev spaces. In presence of additional symmetries, such as radial symmetry, cocompactness may even yield compact embedding, such as in the case of subspaces of radial functions. For details on derived cocompact embeddings we refer to \cite{survey}.

Let $\lambda\in\R$ and let 
\begin{equation} 
 \label{cco}
G_\lambda=\{g_{j,y}:\;u(x)\mapsto 2^{j\lambda}u(2^j(x-y)), j\in\Z, y\in\R^N\}
\end{equation} 
 with $\lambda={\frac{N-p}{p}}=N/p^*$. On either of  $\sob$ and $\leb$, $1<p<N$, this deffines a group of (bijective) linear isometries, often called the rescalings group. In what follows the index $\lambda$ will be usually omitted.

In this note we give four different proofs of the following statement.
\begin{theorem}
\label{thm1}
The Sobolev embedding $\sobemb$, $1<p<N$, is cocompact relative to the rescsaling group $G_{\frac{N-p}{p}}$.
\end{theorem}
It could be possible also to furnish a proof based on atomic decomposition (for definition see \cite{FJW}), characterization of the norm in $B^{-N/p^*,\infty,\infty}$ in terms of the atomic coefficients, and the inequality \eqref{bes-emb}, but given that the atomic decomposition is derived from the Littelwood-Paley decomposition, such proof would be too repetitive of the fourth proof here. 

Throughout this paper we assume that  $N>p$ and $p\in(1,\infty)$, unless stated otherwise.
We will denote  an open ball of radius $r$ centered at $d$ as $B_r(x)$.
We will use the following notations for the conjugate exponent and the critical exponent, resoectively: $p^*=\frac{pN}{N-p}$, $p'=\frac{p}{p-1}$. The quasinorm of the Lorentz space $L^{p,q}(\R^N)$ (for definition see \cite{Triebel}) will be denoted as $\|u\|_{p,q}$. The norm of the usual Lebesgue space  $L^{p}(\R^N)$, identified with the Lorentz space $L^{p,p}(\R^N)$, will be denoted as  
$\|u\|_{p}$.

\section{Capacity-type argument}
The proof below is given by Sergio Solimini in \cite{Solimini}, Section 2. 
Cocompactness in $\sobemb$ is derived there from cocompactness in the embedding into the larger space, $\sob\hookrightarrow L^{p^*,\infty}(\R^N)$, whose equivalent norm, due to the original characterization of the space by Marcinkiewicz, is  
\[
\|u\|_{p^*,\infty}=\sup_{E\subset\R^N}|E|^{-1/p^{*'}}\int_E|u|.
\]
The proof needs the following auxilliary statement.
\begin{lemma}\label{S1}
Let $r>0$, and let
\[
\delta_r(u)(z)=\fint_{B_r(z)} \left(u(x)-u(z)\right)dx.
\]
There exist $C>0$ such that for every $r>0$  every $u\in \sob$, 

\begin{equation}\label{eq:S1}
\|\delta_r(u)\|_p\le C r\|\nabla u\|_p
\end{equation}
\end{lemma}

\begin{proof}
It suffices to prove (\ref{eq:S1}) for $u\in C_0^\infty(\R^N)$. Writing the function $u$ in polar coordinates $\rho,\omega$ centered at a point $z\in\R$, we have 
\[
u(\rho,\omega)-u(0)=\int_0^\rho \partial_t u(t,\omega)dt.
\]
Integrating with respect to $\omega$ over the unit sphere we easily arrive at
\begin{equation}
\left|\int_{B_r(z)} \left(u(x)-u(z)\right)dx\right|\le C\int_{B_r(z)}\frac{|\nabla u(x)|}{|x-z|^{N-1}}dx.
\end{equation}
Inequality (\ref{eq:S1}) follows then immediately by applying Young inequality for convolutions, once we note that the $L^1$-norm of the restriction of $1/|x|^{N-1}$ to $B_r(0)$ is a multiple of $r$. 
\end{proof}
We can now prove Theorem~\ref{thm1}.
\begin{proof}1. It suffices to prove cocompactness of the embedding $\sob\hookrightarrow L^{p^*,\infty}(\R^N)$.
Indeed, there is a continuous embedding $\sob\hookrightarrow L^{p^*,p}(\R^N)$ (expressed by the well-known Hardy inequality), and from application of H\"older inequality to the definition of Lorentz quasinorm one has
\begin{equation}
\|u\|_{p^*}^{p^*}\le C\|u\|_{p^*,p}^p\|u\|_{p^*,\infty}^{p^*-p}\le C\|\nabla u\|_p^p\|u\|_{p^*,\infty}^{p^*-p}. 
\end{equation} 
Thus any sequence bounded in $\sob$ and vanishing in $L^{p^*,\infty}(\R^N)$ will also vanish in $\leb$.
\par
2. Consider a bounded sequence $(v_k)\subset\sob$ which does not converge to zero in the norm of $L^{p,\infty}$. Passing to a renamed subsequence, we may assume that $\|v_k\|_{p,\infty}\ge 4\epsilon$ for some $\epsilon>0$. Then, by definition of the Marcinkiewicz norm, there exist a sequence of measurable sets $Y_k\subset\R^N$ such that 
\[
\int_{Y_k}|v_k|\ge 2\epsilon|Y_k|^\frac{1}{p^{*'}}. 
\] 
Let us choose integers $j_k$ so that the measure of the sets $X_k=2^{j_k}Y_k$ falls into the interval $[1,2^N]$ and let $u_k=g_{j_k,0}v_k$. Note that $\|u_k\|_{p,\infty}=\|v_k\|_{p,\infty}\ge 4\epsilon$, and 
\[
\int_{X_k}|u_k|\ge 2\epsilon. 
\] 
From Lemma~\ref{S1} it now follows that there exists a $r>0$ such that 
\[
\int_{X_k}|\delta_r(u_k)|\le \epsilon. 
\]
Comparing the last two relations and using the triangle inequality we have
\[
\int_{X_k}\left|\fint_{B_r(z)}u_kdx\right|dz\ge \epsilon, 
\]
and thus there exist points $z_k\in X_k$ such that 
\[
\left|\fint_{B_r(z_k)}u_kdx\right|\ge \epsilon 2^{-1-N}.  
\]
In other words,
\[
\int_{B_r(0)}u_k(\cdot-z_k)dx\not\to 0.
\] 
This implies that $u_k(x-z_k)$ does not converge weakly to zero, or, in terms of the original sequence,
$2^{\frac{N-p}{p}j_k}v_k(2^{j_k}(\cdot-y_k))$ does not converge weakly to zero. This contradiction proves that the embedding of $\sob$ into $L^{p^*,\infty}$ is cocompact, and, subsequebntly, the embedding $\sobemb$ is cocompact.
\end{proof}
\section{Proof by partition of domain and range}
Cocompactness of the embedding $\dot H^{1,2}(\R^N)\hookrightarrow L^{2^*}(\R^N)$ for $N>2$ is also proved by Lemma~5.10 in \cite{ccbook}, and its argument can be easily modified for general $p\in(1,N)$ giving another proof of Theorem~\ref{thm1}. 
\begin{proof} Let 
	We may assume without loss of generality that $u_k\in C_0^\infty(\R^N)$.
Let $(u_k)\subset\sob$ and assume that for any $(j_k)\subset\Z$ and any $(y_k)\subset\R^N$, $g_{j_k,y_k}u_\rightharpoonup 0$. Let
$\chi\in C_0^\infty((\frac12,4),[0,3])$, such that
$|\chi^\prime|\le 2$ for all $t$ and  $\chi(t)=t$ for $t\in[1,2]$.
By continuity of the embedding $\sobemb$, we have for every $y\in\Z^N$,
\begin{equation*}
\left(\int_{(0,1)^N+y}\chi(|u_k|)^{p^*}\right)^{p/p^*}\le
C\int_{(0,1)^N+y}(|\nabla u_k|^p+\chi( |u_k|)^p),
\end{equation*}
from which follows, if we take into account that $\chi(t)^{p^*}\le
C t^p$ for $t\ge 0$,
\begin{eqnarray*}
&&\int_{(0,1)^N+y}\chi(|u_k|)^{p^*}\\
&\le &
C\int_{(0,1)^N+y}(|\nabla u_k|^p+\chi(u_k)^p)
\left(\int_{(0,1)^N+y}\chi(|u_k|)^{p^*}\right)^{1-p/p^*}
\\&\le&
C\int_{(0,1)^N+y}(|\nabla
u_k|^p+\chi(|u_k|)^p)\left(\int_{(0,1)^N+y}|u_k|^p\right)^{1-p/p^*}.
\end{eqnarray*}
Adding the above inequalities over $y\in\Z^N$ and taking into
account that $\chi(t)^p\le Ct^{p^*}$ for $t\ge 0$, so that 
\begin{equation*}
\int_{\R^N}\chi(|u_k|)^p\le C\left(\int_{\R^N}|\nabla
u_k|^p\right)^{p^*/p}\le C,
\end{equation*}
we get
\begin{equation}
\label{band1} \int_{\R^N}\chi(|u_k|)^{p^*}\le
C\sup_{y\in\Z^N}\left(\int_{(0,1)^N+y}|u_k|^p\right)^{1-p/p^*}.
\end{equation}
Let $y_k\in\Z^N$ be such that
\begin{equation*}
\sup_{y\in\Z^N}\left(\int_{(0,1)^N+y}|u_k|^p\right)^{1-p/p^*}\le 2
\left(\int_{(0,1)^N+y_k}|u_k|^p\right)^{1-p/p^*}.
\end{equation*}
Since $u_k(\cdot-y_k)\rightharpoonup 0$ in $\sob$ and by the local compactness of subcritical Sobolev embeddings,
\begin{equation*}
\int_{(0,1)^N+y_k}|u_k|^p=\int_{(0,1)^N}|u_k(\cdot-y_k)|^p\to 0.
\end{equation*}
Substituting this into (\ref{band1}), we get
\begin{equation*}
\int_{\R^N}\chi(|u_k|)^{p^*}\to 0.
\end{equation*}
Let
\begin{equation*}
\chi_j(t)=2^j\chi(2^{-j}t)),\, j\in\Z.
\end{equation*}
Note that we may substitute for the original sequence $u_k$ a sequence
$g_{j_k,0}u_k0$, with arbitrary $j_k\in\Z$, and so we have
\begin{equation}
\label{band3} \int_{\R^N}\chi_{j_k}(|u_k|)^{p^*}\to 0.
\end{equation}
Note now that, with $j\in\Z$,
\begin{equation*}
\left(\int_{\R^N}\chi_{j}(|u_k|)^{p^*}\right)^{p/p^*}\le
C\int_{2^{j-1}\le|u_k|\le 2^{j+2}} |\nabla u_k|^p,
\end{equation*}
which can be rewritten as
\begin{equation}
\label{bands2} \int_{\R^N}\chi_{j}(|u_k|)^{p^*}\le
C\int_{2^{j-1}\le|u_k|\le 2^{j+2}} |\nabla
u_k|^p\left(\int_{\R^N}\chi_{j}(|u_k|)^{p^*}\right)^{1-\frac{p}{p^*}}.
\end{equation}
Adding the inequalities (\ref{bands2}) over $j\in\Z$ and taking
into account that the sets $2^{j-1}\le|u_k|\le 2^{j+2}$ cover $\R^N$
with a uniformly finite multiplicity, we obtain
\begin{equation}
\label{bands3} \int_{\R^N}|u_k|^{p^*}\le C\int_{\R^N} |\nabla
u_k|^p\sup_{j\in\Z}\left(\int_{\R^N}\chi_{j}(|u_k|)^{p^*}\right)^{1-p/p^*}.
\end{equation}
Let $j_k$ be such that
\begin{equation*}
\sup_{j\in\Z}\left(\int_{\R^N}\chi_{j}(|u_k|)^{p^*}\right)^{1-p/p^*}\le
2\left(\int_{\R^N}\chi_{j_k}(|u_k|)^{p^*}\right)^{1-p/p^*},
\end{equation*}
and note that the right hand side converges to zero due to
(\ref{band3}). Then from (\ref{bands3}) follows that $u_k\to 0$ in
$L^{p^*}$, which yields the cocompactness.
\end{proof}
\section{Proof by the wavelet decomposition}
We give here the proof of cocompactness in Sobolev embeddings by Stephane Jaffard, \cite{Jaffard}. Note that Jaffard's result covers spaces $\dot H^{s,p}$ for general $s$, and that it has been further extended in  \cite{BCK} to a greater range of embeddings involving Besov and Triebel-Lizorkin spaces, which also admit an unconditional wavelet basis of rescalings of a mother wavelet.
This proof is the shortest in this survey, because the hard analytic part if the argument relies on the wavelet analysis. Specifically, we use the following results (see \cite{Daubechies, Meyer} for details)
There exists a function $\psi:\R^N\to\R$, called {\em mother wavelet}, such that the family $\{2^{\frac{N-p}{p}j}\psi(2^j-k);\,j\in\Z, k\in\Z^N\}\subset \{g\psi, g\in G_\frac{N-p}{p}\}$ forms an unconditional basis in $\sob$ as well as (among others) in the Besov space $\dot B^{-N/p^*,\infty,\infty}$. Moreover, if  $c_{j,k}[u]$ are coefficients in the expansion of $u$ in this basis, one has the following equivalent norm:
\begin{equation}\label{Bes}
\|u\|_{\dot B^{-N/p^*,\infty,\infty}}=\sup_{j\in\Z,k\in\Z^N}|c_{j,k}[u]|.
\end{equation}
We can now prove Theorem~\ref{thm1}.
\begin{proof}
The starting point of the proof is the following inequality (see \cite{GMO}):
\begin{equation}\label{bes-emb}
\|u\|_{p^*}\le C\|u\|_{\sob}^{p/p^*}\|u\|_{\dot B^{-N/p^*,\infty,\infty}}^{1-p/p^*}.
\end{equation}
Assume that for any $(j_n)\subset\Z$ and any $(y_n)\subset\R^N$, $g_{j_n,y_n}u_n \rightharpoonup 0$ in $\sob$.
Since $c_{j,k}[u]$ are continuous functionals on $\sob$ and since for any $(j,k),(j',k')\in\Z\times\Z^{N})$ there exists
a $g\in G$ such that $c_{k,j}[gu]=c_{k',j'}[u]$, it follows that for any $(j_n,k_n)\in\Z\times\Z^N$, $c_{j,k}[u_n] \to 0$ as $n\to\infty$, and, consequently, using \eqref{Bes}, we have
 \begin{equation}\label{b3}
\|u_n\|_{\dot B^{-N/p^*,\infty,\infty}}= \sup_{j\in\Z,k\in\Z^N}|c_{j,k}[u_n]| \to 0.  
 \end{equation}
Then, by \eqref{bes-emb}, $u_n\to 0$ in $L^{p^*}$.
\end{proof}
\section{Proof via the Paley-Littlewood decomposition}
The proof below is a modification of the argument given in \cite{KiVi},Chapter 4: we replaced their intermediate step, inequality (4.16) of Proposition~4.8 in \cite{KiVi}, with a related inequality \eqref{bes-emb} which we already used above. The proof is generalized from the case $p=2$ to the case $p\in(1,N)$.

For details concerning the Littlewood-Paley theory we refer reader the to  Chapter 5 of \cite{Triebel} or Chapter 5 of \cite{FJW}. 
The Paley-Littlewood family of operators $\{P_j\}_{j\in\Z}$ is based on existence of a family of functions $\{\varphi_{j}\}_{j\in\Z}$
with the following properties:
\begin{equation}
\mathrm{supp}\,\{\varphi_{j}\subset\left\{ \xi\in\mathbb{R}^{N},2^{j-1}\le\left|\xi\right|\le2^{j+1}\right\} \label{eq:ayz-1}
\end{equation}
\begin{equation}
\sum_{j\in\Z}\varphi_{n}(\xi)=1\mbox{ for all }\xi\in\mathbb{R}^{N}\setminus\left\{ 0\right\} .\label{eq:byz-1}
\end{equation}
\begin{equation}
\varphi_{j}(\xi)=\varphi_{0}(2^{-j}\xi)\mbox{ for all }\xi\in\mathbb{R}^{N}\mbox{ and }j\in\mathbb{Z}.\label{eq:ccyz-1}
\end{equation}
\begin{equation}
\varphi_{j-1}(\xi)+\varphi_{j}(\xi)+\varphi_{j+1}(\xi)=1\mbox{ for all }\xi\in\mathrm{supp}\varphi_{j}\,.\label{eq:fslz-1}
\end{equation} 
Then $P_j:\sob\to \R$, $j\in \Z$, are given by
$$
P_ju= F^{-1}\varphi_{0}(2^{-j}\cdot)Fu,
$$
where $F$ is the Fourier transform, say, normalized as a
unitary operator in $L^{2}(\mathbb{R}^{N})$. 
In what follows, we will use the following equivalent norm of $\bes$ (cf (4) in Chapter 5 of \cite{Triebel} or (5.2) in \cite{FJW}): 
\begin{equation}\label{Bes2}
\|u\|_{\dot B^{-N/p^*,\infty,\infty}}=\sup_{j\in\Z}\|2^{-\frac{N}{p^*}j}P_{j}u\|_\infty.
\end{equation}
We can now prove Theorem~\ref{thm1}.
\begin{proof}Assume that for any $(j_n)\subset\Z$ and any $(y_n)\subset\R^N$, $g_{j_n,y_n}u_n\rightharpoonup 0$ in $\sob$. By \eqref{bes-emb} it suffices to show that $u_k\to 0$ in $\bes$. By  \eqref{Bes2}, it suffices then to show that 
\begin{equation}
\label{a1}
\sup_{j\in\Z}\|2^{-\frac{N}{p^*}j}P_{j}u_k\|_\infty\to 0.
\end{equation}
or, equivalently, that for any sequence $(j_k)\subset\Z$,
\begin{equation}
\label{a2}
\|2^{-\frac{N}{p^*}j_k}P_{j_k}u_k\|_\infty\to 0\to 0,
\end{equation}
or, equivalently, setting $v_k=g_{j_k,0}u_k$ and noting that  $v_k(\cdot-y_k)\rightharpoonup 0$ for any sequence $(y_k)\subset\R^N$,
\begin{equation}
\label{a3}
\|P_{0}v_k\|_\infty\to 0.
\end{equation}
Choose any $s_0>N/p$, so that  $\dot H^{s_0,p}(\R^N)$ is compactly embedded into $C(\R^N)$. Since $P_0v_k$ is bounded in $\dot H^{s_0,p}(\R^N)$, we have 
$v_k(\cdot-y_k)\rightharpoonup 0$  in $\dot H^{s_0,p}(\R^N)$ for any sequence $(y_k)\subset\R^N$, and \eqref{a3} follows by compactness of the embedding into $C(\R^N)$. This implies \eqref{a1}, and by \eqref{Bes2} $u_k\to 0$ in $\leb$. 
\end{proof}

\bibliographystyle{amsplain}

\end{document}